\documentclass[10pt]{amsart}

\newtheorem{theoreme}{Theorem}[section]

\newtheorem{lemme}[theoreme]{Lemma}

\newtheorem{proposition}[theoreme]{Proposition}

\newtheorem{remarque}[theoreme]{Remark}

\newtheorem{corollaire}[theoreme]{Corollary}

\input xy

\xyoption{all}

\usepackage{amscd}

\usepackage{amssymb}
\usepackage{fancyhdr}

\usepackage[a4paper, left=2cm,right=2cm,bottom=2cm,top=3.5cm]{geometry}
\usepackage{amsthm,array,amssymb,amscd,amsfonts,latexsym, url}
\usepackage{amsmath}
\usepackage{graphicx,epsfig}

\title{On the universal $\mathrm{CH}_0$ group of cubic threefolds in positive characteristic}
\author{Ren\'e Mboro}
\date{}
\begin{document}
\maketitle
%\begin{abstract}
%\scriptsize{
%\begin{center} 2000 \textit{Mathematics Subject Classification.} Primary: ; Secondary: .\end{center}}
%\end{abstract}\textit{}\\
%\tableofcontents
\begin{abstract}
%\scriptsize{
%\begin{center} 2000 \textit{Mathematics Subject Classification.} Primary: ; Secondary: .\end{center}}
We adapt for algebraically closed fields $k$ of characteristic greater than $2$ two results of Voisin, presented in \cite{Vois_main}, on the decomposition of the diagonal of a smooth cubic hypersurface $X$ of dimension $3$ over $\mathbb C$, namely: the equivalence between Chow-theoretic and cohomological decompositions of the diagonal of those hypersurfaces and the equivalence between the algebraicity (with $\mathbb Z_2$-coefficients) of the minimal class $\theta^4/4!$ of the intermediate Jacobian $J(X)$ of $X$ and the cohomological (hence Chow-theoretic) decomposition of the diagonal of $X$.  Using the second result, the Tate conjecture for divisors on surfaces defined over finite fields predicts, via a theorem of Schoen (\cite{Schoen}), that every smooth cubic hypersurface of dimension $3$ over the algebraic closure of a finite field of characteristic $>2$ admits a Chow-theoretic decomposition of the diagonal.
\end{abstract}

\section{Introduction}
Let $k$ be an algebraically closed field and $X$ a smooth $n$-dimensional projective $k$-variety (irreducible). X is said to have universally trivial $\mathrm{CH}_0$ group if for any field $L$ containing $k$, $\mathrm{CH}_0(X_L)=\mathbb Z$. As explained in \cite{auel_colliot}, $X$ has universally trivial $\mathrm{CH}_0$ group if and only if $\mathrm{CH}_0(X_{k(X)})=\mathbb Z$. We recall briefly the idea of the proof; it uses, for any extension $k\subset L$, the action, by the identity, of the correspondence $\Delta_{X_L}$ on $\mathrm{CH}_0(X_L)$ and the equivalence of the two previous properties with a so called Chow-theoretic decomposition of the diagonal. Passing to the limit in the diagonal morphism $X\rightarrow X\times_k X$ over all open subsets $V\subset X$ ($V\rightarrow V\times_k V\hookrightarrow V\times_k X$) yields the diagonal point $\delta_X\in X(k(X))$ (image of the generic point of $X$ by the diagonal morphism). When $\mathrm{CH}_0(X_{k(X)})=\mathbb Z$, the diagonal point is rationally equivalent over $k(X)$ to a constant point $x_{k(X)}=x\times_k k(X)$, with $x\in X(k)$. Then, using the equalities $\mathrm{CH}_0(X_{k(X)})=\mathrm{CH}^n(X_{k(X)})=\varinjlim_{\atop U\subset X\ open} CH^n(U\times_k X)$ and the localization exact sequence, one gets an equality
\begin{equation}\label{ch_decomp_def}
\Delta_X = X\times_k x + Z\ \mathrm{in\ CH}^n(X\times_k X)
\end{equation}
where $Z$ is supported on $D\times_k X$ for some proper closed subset $D$ of $X$ i.e. a Chow theoretic decomposition of the diagonal. Now, letting $L/k$ be an extension, a base change in (\ref{ch_decomp_def}) yields $$\Delta_{X_L} = X_L\times_L x_L + Z_L\ \mathrm{in\ CH}^n(X_L\times_L X_L)$$ so that letting both sides act on $\mathrm{CH}_0(X_L)$ and using the fact that $\Delta_{X_L}$ acts as the identity, one sees that $\mathrm{CH}_0(X_L)=\mathbb Z$. So, having a universally trivial $\mathrm{CH}_0$ is equivalent to the existence of a Chow-theoretic decomposition of the diagonal. Projective spaces have universally trivial $\mathrm {CH}_0$ so stably rational projective varieties also have universally trivial $\mathrm {CH}_0$. Voisin studied in the case $k=\mathbb C$ the a priori weaker property of the existence of a cohomological decomposition of the diagonal
\begin{equation}
[\Delta_X]=[X\times_k x] + [Z]\ in\ H_B^{2n}(X\times_k X,\mathbb Z)
\end{equation}
where $Z\in \mathrm{CH}^n(X\times_k X)$ is supported on $D\times_k X$ for some proper closed subset $D$ of $X$ and the cohomology is the Betti cohomology. In \cite{Vois_main}, she proved, in characteristic zero, that this weaker property is, in fact, equivalent to the existence of a Chow-theoretic decomposition of the diagonal for smooth cubic hypersurfaces of odd dimension ($\geq3$) or of dimension $4$. She also worked out necessary and sufficient conditions for the existence of a cohomological (hence Chow-theoretic) decomposition of the diagonal for cubic threefolds, among other varieties. The second section of this paper is devoted to the proof of the equivalence between cohomological and Chow-theoretic decomposition of the diagonal of a cubic threefold in positive characteristic, greater than $2$:

\begin{theoreme}\label{coh-chow} Let $k$ be an algebraically closed field of characteristic greater than $2$ and $X\subset \mathbb P^4_k$ be a smooth cubic hypersurface. Then $X$ admits a Chow-theoretic decomposition of the diagonal (i.e. has universally trivial $\mathrm{CH}_0$) if and only if it admits a cohomological decomposition of the diagonal with coefficients in $\mathbb Z_2$.
\end{theoreme}

\begin{remarque} (1) In the Betti setting, having a cohomological decomposition of the diagonal with $\mathbb Z$ coefficients is equivalent to having a cohomogical decomposition of the diagonal with coefficients in $\mathbb Z/2\mathbb Z$ since $2\Delta_X$ has a Chow-theoretic decomposition (see Proposition \ref{unirat}). This is why in our case, only \'etale cohomology with $\mathbb Z_2$ coefficients is used. In fact, $\mathbb Z/2\mathbb Z$ would also do.\\
\indent \indent \indent \indent \indent (2) Theorem \ref{coh-chow} could have been stated also for cubic fourfolds and odd dimensional cubic of higher dimension since the proof given by Voisin adapts to the positive characteristic setting. But since in the case of cubic threefolds we can give a shorter proof and it is the only case where we have an interesting application in sight, we state the theorem only in this case.
\end{remarque}
We also check that there is a criterion, similar to the one given in \cite[Theorem 4.1]{Vois_main} and \cite[Theorem 4.9]{Vois_abel} in characteristic zero, for the existence of a cohomological decomposition of the diagonal of a cubic threefold.
\begin{theoreme}\label{criterion} Let $X\subset \mathbb P^4_k$ be a smooth cubic hypersurface ($k=\bar k$ and $char(k)>2$). Then $X$ admits a cohomological (hence Chow-theoretic) decomposition of the diagonal if and only if the principal polarization, $\theta$, of the associated Prym variety, $J(X)$, satisfies the following property: there is a cycle $Z\in \mathrm{CH}_1(J(X))\otimes \mathbb Z_2$ such that $[Z]=\frac{[\theta]^4}{4!}$ in $H^8(J(X),\mathbb Z_2)$.
\end{theoreme}

%In the above theorem, the direction: if $\frac{[\theta]^4}{4!}$ is algebraic then there is a decomposition of the diagonal, is still true for characteristic $3$ and $5$.
Using Theorem \ref{criterion} and a theorem of C. Schoen (see \cite{colliot_sz}, \cite{Schoen}), we get the following consequence:
\begin{theoreme}\label{cor_intro} On an algebraic closure of a finite field of characteristic greater than $2$, assuming the Tate conjecture for divisors on surfaces, every smooth cubic hypersurface of dimension $3$ has universally trivial $\mathrm{CH}_0$ group.\\
\end{theoreme}

In the Betti setting, a key feature in the proof of the criterion was the existence of a parametrization of the intermediate Jacobian of cubic threefolds with separably rationally connected general fiber, namely the condition:\\

\indent (*) there exist a smooth quasi-projective $k$-variety $B$ and a correspondence $Z\in \mathrm{CH}^2(B\times_k X)$ with $Z_b\in \mathrm{CH}^2(X)$ trivial modulo algebraic equivalence for any $b\in B(k)$ such that the induced Abel-Jacobi morphism $\phi:B\rightarrow J(X)$ to the intermediate Jacobian $J(X)$ of the cubic threefold $X$ is dominant with $\mathbb P^5$ as general fiber.\\

\indent In Section $3$, after a reminder on the definition of the intermediate Jacobian of a cubic threefold and Abel-Jacobi morphisms in the positive characteristic setting, we prove Thereom \ref{criterion} and Theorem \ref{cor_intro} under the assumption that (*) is still true in our setting.\\
\indent Over the complex numbers, such a parametrization was constructed by Iliev-Markushevich and Markuschevich-Tikhomirov (\cite{Mar-Tik} and \cite{Ili-Mar}, see also \cite{Druel}) using the space of smooth normal elliptic quintics lying on the cubic hypersurface. Section $4$ will be devoted to proving that we still have (*) using the space of stable normal elliptic quintics.\\

\indent For a variety $X$, the cohomology groups $H^i(X,\mathbb Z_{\ell})$ will be the \'etale cohomology groups and $H^i_B(X,\mathbb Z_{\ell})$, if $X$ is defined over a field $K\hookrightarrow \mathbb C$ of characteristic $0$, will be the Betti cohomology group of $X^{an}_{\mathbb C}$.\\
\indent Throughout this text, $k$ will denote an algebraically closed field of characteristic $>2$.\\
\indent For a smooth projective $k$-variety $X$, $\langle\cdot,\cdot\rangle_X$ will denote the intersection pairing on $H^{*}(X,\mathbb Z_{\ell})$ induced by the cup-product and the trace map. We will use the following standard facts about cubic hypersurfaces in $\mathbb P^4_k$. Let $X$ be a smooth cubic hypersurface in $\mathbb P^4_k$.
\begin{enumerate}\label{reminder_on_cubic}
\item From the exact sequence $$0\rightarrow \mathcal O_{\mathbb P^4_k}(-3) \rightarrow \mathcal O_{\mathbb P^4_k} \rightarrow \mathcal O_X\rightarrow 0$$ we have $\omega_X\simeq \mathcal O_X(-2)$ and $h^i(\mathcal O_X(k))=0$ for $i\in\{1,2\},$ $k\in \mathbb Z.$\\
\item The Fano variety of lines $F(X)=\{[l]\in Gr(2,5),\ l\subset X\}$ of $X$ is a smooth projective surface.\\
\item By Lefschetz hyperplane theorem, for $\ell\neq p$ the $\ell$-adic cohomology of $X$ is $$H^0(X,\mathbb Z_{ \ell}) = \mathbb Z_{\ell}\cdot[X], \ \ H^2(X,\mathbb Z_{\ell}(1)) = \mathbb Z_{\ell}\cdot[\mathcal O_X(1)], \ \ H^4(X,\mathbb Z_{\ell}(2))=\mathbb Z_{\ell}\cdot[l], \ \ H^6(X, \mathbb Z_{\ell}(3)) = \mathbb Z_{\ell}\cdot[x],$$ $$ H^1(X,\mathbb Z_{\ell})=0=H^5(X,\mathbb Z_{\ell})$$ where $[l]$ is the class of a line $[l]\in F(X).$
\item Using, for example, a smooth proper lifting of $X$ to characteristic zero (\cite[Section 20]{Milne_et_coh}), we have $H^3(X,\mathbb Z_{\ell})\simeq \mathbb Z_{\ell}^{10}.$ By Grothendieck-Lefschetz theorem, $\mathrm{Pic}(X)\simeq \mathbb Z \cdot[\mathcal O_X(1)]$.
\end{enumerate}

\section{Chow-theoretic and $\mathbb Z_2$-cohomological decomposition of the diagonal}

In this section, for a $k$-variety $Y$, $\mathrm{B}^i(Y)$ will designate the Chow group of codimension $i$ cycle modulo algebraic equivalence. We prove in this section Theorem \ref{coh-chow}, adapting arguments of Voisin presented in \cite{Vois_main} to the positive characteristic case. The key point is to prove that one can derive a decomposition of the diagonal modulo algebraic equivalence from a cohomological decomposition of the diagonal $\Delta_X$ of a smooth cubic hypersurface $X$. Then we use the following proposition to obtain a Chow-theoretic decomposition of the diagonal:
\begin{proposition}\label{alg-rat}
Let $X$ be a smooth projective $k$-variety of dimension $n$. Suppose there exists $Z\in \mathrm{CH}^n(X\times_k X)$, supported on $D\times_k X$ for some proper closed subset $D\subset X$, and $x\in X(k)$ such that $$\Delta_X - X\times x = Z\ \mathrm{in\ B}^n(X\times_k X).$$ Then $X$ admits a Chow-theoretic decomposition of the diagonal.
\end{proposition}
\begin{proof} It is proposition 2.1 of \cite{Vois_main} since even in positive characteristic, cycles algebraically equivalent to $0$ are nilpotent for the composition of self-correspondences (see \cite{Vois_nilp} and \cite{Voe} which use no assumptions on $char(k)$).
\end{proof}

We recall the following classical fact on a cubic hypersurface.
\begin{proposition}\label{unirat} Let $X$ be a smooth cubic hypersurface of dimension $3$. Then $X$ admits a degree $2$ dominant rational map $\mathbb P^3\dashrightarrow X$. It follows that $2\Delta_X$ admits a decomposition $$2\Delta_
X = 2(X\times_k x) + Z\ in\ \mathrm{CH}^3(X\times_k X)$$ where $x\in X(k)$ and $Z$ is supported on $D\times_k X$ for some divisor $D\varsubsetneq X.$
\end{proposition}
\begin{proof}[Sketch of proof.]The first fact is classical and is presented, for example in appendix B of \cite{C-G}. We recall briefly the construction of the degree $2$ map from a rational variety. Let $l_0$ be a general line in $X$ then the map $P(T_{X|l_0})\dashrightarrow X$  taking a point $(x,v)$ (with $x\in l_0$ and $v\in P(T_{X,x})$) such that the line $\langle x,v\rangle$ is tangent to $X$ at $x$ but not contained in $X$, to the other point of the intersection $X\cap \langle x,v\rangle$, is generically finite and $2:1$.\\
\indent So we have a rational map $\varphi:\mathbb P^3_k\dashrightarrow X$ of degree $2$. Since resolution of singularities for threefolds exists in $char(k)>0$ by work of Cossart and Piltant (\cite{cos-pil1} and \cite{cos-pil2}), there is a smooth projective $k$-variety $\Gamma$, with a birational morphism $p:\Gamma\rightarrow \mathbb P^3_k$ and a degree $2$ morphism $\phi:\Gamma\rightarrow X$, resolving the indeterminacies of $\varphi$. We have the following lemma:
\begin{lemme}\textit{(\cite[Proposition 6.3]{Colliot-Coray}).} Let $f:Z\rightarrow Y$ be a birational morphism of smooth geometrically integral projective varieties over a field $L$. Then $\mathrm{CH}^0_0(Z)\simeq \mathrm{CH}^0_0(Y)$, where $\mathrm{CH}^0_0(T)$, for a proper $L$-variety $T$, is the group of $0$-cycles of degree $0$.
\end{lemme} 
Applying the lemma to the morphism obtained from $p$ by base change $p_{k(\mathbb P^3_k)}:\Gamma_{k(\mathbb P^3_k)}\rightarrow \mathbb P^3_{k(\mathbb P^3_k)}$ yields $\mathrm{CH}_0(\Gamma_{k(\mathbb P^3_k)})\simeq \mathbb Z$ i.e. $\Gamma$ has universally trivial $\mathrm{CH}_0$ group. Then, by base change we have the morphism $\phi_{k(X)}:\Gamma_{k(X)}\rightarrow X_{k(X)}$. The $0$-cycle of $X_{k(X)}$, $\delta_X - k(X)\times_k x$ has degree $0$, where $\delta_X\in X(k(X))$ is the diagonal point (the image of the generic point of $X$ by the diagonal morphism) and $x\in X(k)$. Since $\phi_{k(X)}$ is a degree $2$ proper morphism, we have $\phi_{k(X),*}\phi_{k(X)}^*(\delta_X - k(X)\times_k x)=2(\delta_X - k(X)\times_k x)$ but since that operation factors through $\mathrm{CH}^0_0(\Gamma_{k(X)})$, which zero, we have $2(\delta_X - k(X)\times_k x)=0$.\\
\end{proof}

For a smooth projective $k$-variety, the second punctual Hilbert scheme $X^{[2]}$ of $X$ is obtained as the quotient of the blow-up $\widetilde{X\times_k X}$ of $X\times_k X$ along the diagonal by its natural involution. Let $\mu:X\times_k X\dashrightarrow X^{[2]}$ be the natural rational map and $r:\widetilde{X\times_k X}\rightarrow X^{[2]}$ be the quotient morphism. We collect some results of \cite{Vois_main} whose proofs are essentially the same. So we just mention, when needed, the change needed or the facts required in characteristic $p>2$:

\begin{lemme}\textit{(\cite[Lemma 2.3]{Vois_main}).} Let $X$ be a smooth projective variety of dimension $n$. Then there exists a codimension $n$ cycle $Z$ in $X^{[2]}$ such that $\mu^{*}Z= \Delta_X$ in $\mathrm{CH}^n(X\times_k X)$.
\end{lemme}

\begin{corollaire}\textit{(\cite[Corollary 2.4]{Vois_main}).} Any symmetric codimension $n$ cycle on $X\times_k X$ is rationally equivalent to $\mu^{*}\Gamma$ for a codimension $n$ cycle $\Gamma$ on  $X^{[2]}$.
\end{corollaire}

\begin{lemme}\label{coh-coh_sym}\textit{(\cite[Lemma 2.5]{Vois_main}).} Let $X$ be smooth projective $k$-variety of dimension $n$. Suppose $X$ admits a $\mathbb Z_2$-cohomological decomposition of the diagonal $$[\Delta_X - x\times_k X] = [Z]\ \text{in}\ H^{2n}(X\times_k X, \mathbb Z_2(n))$$ where $Z$ is a cycle supported on $D\times_k X$ for some proper closed subset $D$ of $X$ and $x$ $k$-rational point of $X$. Then $X$ admits a $\mathbb Z_2$-cohomological decomposition of the diagonal $$[\Delta_X - x\times_k X - X\times_k x] = [W]\ \text{in}\ H^{2n}(X\times_k X, \mathbb Z_2(n)),$$ where $W$ is a cycle supported on $D\times_k X$ and $W$ is invariant under the natural involution of $X\times_k X$.
\end{lemme}
The following result is proved in \cite[Proposition 2.6]{Vois_main} over $\mathbb C$.
\begin{proposition}\label{key_point} Let $X$ be a smooth odd degree complete intersection of odd dimension $n$. If $X$ admits a $\mathbb Z_2$-cohomological decomposition of the diagonal, there exists a cycle $\Gamma\in \mathrm{CH}^n(X^{[2]})$ with the following properties:
\begin{enumerate}
\item $\mu^{*}\Gamma = \Delta_X - x\times_k X - X\times_k x - W$ in $\mathrm{CH}^n(X\times_k X),$ with $W$ supported on $D\times_k X,$ for some closed proper subset $D\subset X$.
\item $[\Gamma] = 0$ in $H^{2n}(X^{[2]},\mathbb Z_2(n))$.
\end{enumerate}
\end{proposition}

\begin{proof}[Sketch of proof of \ref{key_point}] The proof uses an analysis of the cohomology group $H^{2n}(X^{[2]},\mathbb Z_2(n))$ and more precisely of the morphism $j_{E_X *}:H^{2n-2}(E_X,\mathbb Z_2)\rightarrow H^{2n}(X^{[2]},\mathbb Z_2(n))$, where $E_X$ is the exceptional divisor of $X^{[2]}$. This analysis is delicate for even dimensional odd degree complete intersections, but for odd dimension and odd degree complete intersections, the restriction map from the even degree cohomology of projective space to the even degree cohomology of $X$ is surjective, so the result follows from the analysis of the cohomology of $(\mathbb P^{N})^{[2]}$ which is Chow theoretic and works in any characteristic.\\
\indent The only additional fact to check is that the cohomology of $X^{[2]}$ has no $2$-torsion when $X$ is an odd degree complete intersection in projective space. To see this, choose a smooth projective lifting of $X$ to characteristic $0$ over a discrete valuation ring $\mathfrak X\rightarrow Spec(R)$. As a zero dimensional length two subscheme is local complete intersection, and has trivial degree $1$ coherent cohomology, $\mathrm{Hilb}_2(\mathfrak X/Spec(R))\rightarrow Spec(R)$ is smooth and projective by \cite[Proposition I.2.15.4]{rat-cur-kol} (and smoothness of the fibers $\mathfrak X_{\bar{\eta}}^{[2]}$ and $X^{[2]}$). So, $H^{r}(\mathfrak X_{\bar{\eta}}^{[2]},\mathbb Z_2)\simeq H^{r}(X^{[2]}, \mathbb Z_2)\ \forall r\geq 0$ by the smooth proper base change. Now by the comparison theorem, $H^{r}(\mathfrak X_{\bar{\eta}}^{[2]},\mathbb Z_2)\simeq H^{r}_B(\mathfrak X_{\bar{\eta}}^{[2]},\mathbb Z_2)$ and by \cite{Totaro-coh} these last groups have no $2$-torsion. The rest of the proof works just like in \cite{Vois_main}.
\end{proof}

For $X$ a smooth cubic hypersurface in $\mathbb P^{n+1}$, we recall another description of $X^{[2]}$ used in \cite{Vois_main}. Let $F(X)$ be the variety of lines of $X$ and $P=\{([l],x), \ x\in l,\ l\subset X\}$ be the universal $\mathbb P^1$-bundle over $F(X)$ with projections $p:P\rightarrow F(X)$ and $q:P\rightarrow X$, and let $P_2\rightarrow F(X)$ be the $\mathbb P^2$-bundle defined as the symmetric product of $P$ over $F(X).$ There is a natural embedding $P_2\stackrel{i_{P_2}}{\hookrightarrow} X^{[2]}$ which maps each fiber of $P_2\rightarrow F(X),$ that is the second symmetric product of a line in $X,$ isomorphically onto the set of subschemes of length $2$ of $X$ contained in this line. Let $p_X:P_X\rightarrow X$ be the projective bundle with fiber over $x\in X$ the set of lines in $\mathbb P^{n+1}$ passing through $x.$ Note that $P$ is naturally contained in $P_X.$

\begin{proposition}\label{sym_gal}\textit{(\cite[Proposition 2.9]{Vois_main}).} In the above situation, we have the following properties:\\
\indent (i) The natural map $\Phi: X^{[2]}\dashrightarrow P_X$ which to a unordered pair of points $x,\ y\in X$ not contained in a common line of $X$ associates the pair $([l_{x,y}],z),$ where $l_{x,y}$ is the line in $\mathbb P^{n+1}$ generated by $x$ and $y$, and $z\in X$ is the residual point of the intersection $l_{x,y}\cap X,$ is desingularized by the blow-up of $X^{[2]}$ along $P_2$.\\
\indent (ii) The induced morphism $\widetilde{\Phi}: \widetilde{X^{[2]}}\rightarrow P_X$ identifies $\widetilde{X^{[2]}}$ with the blow-up $\widetilde{P_X}$ of $P_X$ along $P$.\\
\indent (iii) The exceptional divisors of the two maps $\widetilde{X^{[2]}}\rightarrow X^{[2]}$ and $\widetilde{P_X}\rightarrow P_X$ are identified by the isomorphism $\widetilde{\Phi}':\widetilde{X^{[2]}}\cong \widetilde{P_X}$ of (ii).
\end{proposition}

\begin{proof}[Proof of Theorem \ref{coh-chow}.] Let $X\subset \mathbb P^4_k$ be a smooth cubic  threefold. By Proposition \ref{unirat}, we see that there is a nonempty open subset $U_0\subset X$, such that $(\Delta_X - X\times_k x)_{|U_0\times_k X}$ is a $2$-torsion element of $\mathrm{B}^3(U_0\times_k X)$. The subgroup of $2$-torsion elements of $\mathrm{B}^3(U_0\times_k X)$ is a $\mathbb Z/2\mathbb Z$-module and since $\mathbb Z/2\mathbb Z$ is a quotient of the localization $\mathbb Z_{(2)}$ of $\mathbb Z$ in $2\mathbb Z$, a $2$-torsion element of $\mathrm{B}^3(U_0\times_k X)$ is $0$ if and only if it is $0$ in $\mathrm{B}^3(U_0\times_k X)\otimes \mathbb Z_{(2)}$. Since $\mathbb Z_2$ is the completion of the local ring $\mathbb Z_{(2)}$ along its maximal ideal, $\mathbb Z_2$ is a faithfully flat $\mathbb Z_{(2)}$-module. Hence, a $2$-torsion element in $\mathrm{B}^3(U_0\times_k X)$ is $0$ if and only if it is $0$ in $\mathrm{B}^3(U_0\times_k X)\otimes \mathbb Z_2$. So in order to prove that $X$ admit a Chow-theoretic decomposition of the diagonal, we only need to check that it is $0$ in $\mathrm{B}^3(U'\times X)\otimes_{\mathbb Z} \mathbb Z_2$, for some open subset $U'\subset X$. Once we know that, we will have that $(\Delta_X - X\times_k x)_{|U'\times_k X}$ is $0$ in $B^3(U'\times_k X)$ i.e. a decomposition of the diagonal $$\Delta_X= X\times_k x +Z\  \mathrm{in\ B}^3(X\times_k X)$$ 
where $Z$ is supported on $D\times_k X$ for some proper closed subset $D\varsubsetneq X$. Applying Proposition \ref{alg-rat}, this will yield the Chow-theoretic decomposition of the diagonal. So we shall work with $\mathbb Z_2$ coefficients and, adapting arguments of \cite{Vois_main}, show that a cohomological decomposition of the diagonal with coefficients in $\mathbb Z_2$ implies that $(\Delta_X - X\times_k x)_{|U\times_k X}=0$ in $\mathrm{B}^3(U\times X)\otimes_{\mathbb Z} \mathbb Z_2$ for some nonempty open subset $U$ of $X$.\\
\indent Assume $X$ admits a cohomological decomposition of the diagonal with coefficients in $\mathbb Z_2.$ The assumptions of Proposition \ref{key_point} are satisfied by $X$, since the cohomology of a smooth cubic hypersurface with coefficients in $\mathbb Z_2$ has no torsion and is algebraic in even degree. Using the notation introduced previously, there exists, by Proposition \ref{key_point}, a cycle $\Gamma\in \mathrm{CH}^3(X^{[2]})$ such that
\begin{equation}
\mu^*\Gamma = \Delta_X - x\times_k X - X\times_k x - W\ \mathrm{in\ CH}^3(X\times_k X),
\label{decomp_init}
\end{equation}
 with $W$ supported on $D\times_k X,$ $D\varsubsetneq X,$ and $[\Gamma]=0$ in $H^{6}(X^{[2]}, \mathbb Z_2(3)).$\\
\indent By Proposition \ref{sym_gal}, the blow-up $\sigma:\widetilde{X^{[2]}}\rightarrow X^{[2]}$ of $X^{[2]}$ along $P_2$ identifies via $\widetilde{\Phi}$ with the blow-up $\widetilde{P_X}$ of $P_X$ along $P$. Furthermore, the exceptional divisor $E\stackrel{i_E}{\hookrightarrow}\widetilde{X^{[2]}}$ of $\widetilde{\Phi}:\widetilde{X^{[2]}}\rightarrow P_X$ is also the exceptional divisor of $\sigma:\widetilde{X^{[2]}}\rightarrow X^{[2]},$ hence maps via $\sigma$ to $P_2\subset X^{[2]}.$ Since $\widetilde{\Phi}$ is a blow-up of a smooth subvariety, the Chow groups of $\widetilde{X^{[2]}}$ decomposes as $\mathrm{CH}^*(\widetilde{X^{[2]}})=\widetilde{\Phi}^*\mathrm{CH}^*(P_X) \oplus i_{E *}\mathrm{CH}^*(E)$; we have a similar decomposition for the cohomology groups $H^*(\widetilde{X^{[2]}},\mathbb Z_2)=\widetilde{\Phi}^*H^*(P_X,\mathbb Z_2)\oplus i_{E *}H^*(E,\mathbb Z_2)$ and these decompositions are compatible with the cycle map. By work of Shen \cite[Theorem 1.1]{Shen}, the group of $1$-cycles of $X$ is generated by lines of $X$ i.e. the action of correspondence $P$ induces a surjective morphism $P_*:\mathrm{CH}_0(F(X))\rightarrow \mathrm{CH}_1(X)$.\\
\indent So let $\gamma$ be a $1$-cycle homologically trivial with coefficients in $\mathbb Z_2$ on $X$; we can write it $P_*(z)$ for a $z\in \mathrm{CH}_0(F(X))\otimes \mathbb Z_2$. The degree of $\gamma\cdot H$, where $H=c_1(\mathcal O_X(1))$, is $0$ in $\mathbb Z_2$ since it is algebraically trivial in $\mathrm{CH}_0(X)\otimes \mathbb Z_2\stackrel{deg}{\simeq}\mathbb Z_2$. But, with the above notations $\gamma\cdot H= P_*(z)\cdot H=q_*p^*(z\cdot q^*H)$ and $q^*H$ is the relative hyperplane divisor of the projective bundle $p:P\rightarrow F(X)$; so that the degree of $z$ is also $0$ in $\mathrm{CH}_0(F(X))\otimes  \mathbb Z_2$ i.e. $z$ is algebraically trivial. So $\gamma= P_*(z)$ is algebraically trivial in $\mathrm{CH}_1(X)\otimes \mathbb Z_2$. So algebraic and homological equivalences with coefficients in $\mathbb Z_2$ coincide on $\mathrm{CH}_1(X)$. Since these relations coincide also on $\mathrm{CH}_2(X)=Pic(X)$ and $\mathrm{CH}_0(X)$, they coincide on the Chow ring of $X$.\\
\indent Then, since $P_X$ is a projective bundle over $X$, the two equivalence relations coincide also on $P_X$. On the other hand, $F(X)$ being a surface, algebraic and homological equivalences coincide on $F(X)$ hence also on the projective bundle $P$ over $F(X)$. Since $E$ is also a projective bundle over $P$, the two equivalence relations coincide on the blow-up $\widetilde P_X$ of $P_X$ along $P$, which is isomorphic to $\widetilde X^{[2]}$.\\
\indent Now $\sigma^*[\Gamma]$ is $0$ in $H^6(\widetilde X^{[2]},\mathbb Z_2(3))$ since $[\Gamma]=0$ so that $\sigma^*\Gamma=0$ in $B^3(\widetilde X^{[2]})\otimes \mathbb Z_2$. We conclude that $\Gamma=0$ in $B^3(X^{[2]})\otimes \mathbb Z_2$ since $\sigma_*\sigma^*=id_{\mathrm{CH}^*(X^{[2]})}$. So (\ref{decomp_init}) yields 
$$\Delta_X = x\times_k X + X\times_k x - W'\ in\ \mathrm{B}^3(X\times_k X)\otimes \mathbb Z_2$$ 
where $W'$ is supported on $D'\times_k X$ for some $D'\varsubsetneq X$. So $(\Delta_X - X\times_k x)_{|U\times_k X}=0$ in $\mathrm{B}^3(U\times_k X)\otimes \mathbb Z_2$ (with $U=X\backslash D'$) as we wanted.

\end{proof}

\section{Cohomological decomposition of the diagonal}
We prove in this section Theorem \ref{criterion}, again adapting arguments of Voisin presented in \cite{Vois_main}. We begin by the following theorem which was proved in \cite[Theorem 3.1]{Vois_main} over $\mathbb C$.

\begin{theoreme}\label{tech_var_aux} Let $X$ be a smooth projective $k$-variety of dimension $n>0$ and $\ell\neq p$ a prime number.
\begin{enumerate}
\item Assume $H^{*}(X,\mathbb Z_{\ell})$ has no torsion, $H^{2i}(X,\mathbb Z_{\ell}(i))$ is algebraic for $2i\neq n$, $H^{2i+1}(X, \mathbb Z_{\ell})=0$ for $2i+1\neq n$ and that $X$ satisfies the following condition:\\
\indent (*) There exist finitely many smooth projective varieties $Z_i$ of dimension $n-2,$ correspondences $\Gamma_i\in \mathrm{CH}^{n-1}(Z_i\times_k X),$ and $n_i\in \mathbb Z_2,$ such that for any $\alpha,\beta\in H^{n}(X,\mathbb Z_{\ell}),$
\begin{equation}
\langle\alpha,\beta\rangle_X = \sum_i n_i\langle\Gamma_i^*\alpha,\Gamma_i^*\beta\rangle_{Z_i}.
\label{etoile}
\end{equation}
Then $X$ admits a cohomological decomposition of the diagonal with coefficients in $\mathbb Z_\ell.$
\item If $n=3$, $char(k)\neq \ell$ and $X$ admits a cohomogical decomposition of the diagonal with coefficients in $\mathbb Z_\ell,$ then (*) is satisfied.
\end{enumerate}
\end{theoreme}
\begin{proof}[Sketch of proof.]The adaptation of the proof given in \cite{Vois_main}, to positive characteristic is straightforward; the only fact to use for the second point is the existence of (embedded) resolution of singularities in dimension $3$ for algebraically closed fields of positive characteristic (see \cite{cos-pil1} and \cite{cos-pil2}).
\end{proof}

\subsection{The intermediate Jacobian of a cubic threefold and Abel-morphisms}\textit{}\\
\indent Let us denote $\mathrm{CH}^2_{alg}(Y)$ the group of codimension $2$ cycles algebraically equivalent to zero on a $k$-variety $Y$. Given an abelian variety $Ab$ over $k$, following \cite[VIa]{Murre_rep}, we shall call a (group) homomorphism $f: \mathrm{CH}^2_{alg}(Y)\rightarrow Ab$ a regular morphism if for any smooth quasi-projective $k$-variety $T$ and $Z\in \mathrm{CH}^2(T\times_k Y)$ such that for any $t\in T(k)$, $Z_t\in \mathrm{CH}^2_{alg}(Y)$, the composition $T\rightarrow \mathrm{CH}^2_{alg}(Y)\stackrel{f}{\rightarrow} Ab$ is a morphism of algebraic varieties. We say that $\mathrm{CH}^2_{alg}(Y)$ admits an algebraic representative if there is an abelain variety $Ab(Y)$ over $k$ and a regular morphism $\phi:\mathrm{CH}^2_{alg}(Y)\rightarrow Ab(Y)$ which is universal in the sense that any regular morphism factor as a composition of $\phi$ followed by a momorphism of algebraic varieties. In that case we call the morphism $\phi_Z:T\rightarrow \mathrm{CH}^2_{alg}(Y)\stackrel{f_Y}{\rightarrow} Ab(Y)$ the Abel-Jacobi morphism induced by $Z$. By \cite[Theorem 1.9]{Mur_app}, $\mathrm{CH}^2_{alg}(Y)$ admits an algebraic representative when $Y$ is a smooth projective variety over an algebraically closed field.\\
%\indent Since $\mathbb P^3_k$ is separably rationally connected and there is a dominant degree $2$ rational map $\mathbb P^3_k\dashrightarrow X$ and $2\neq char(k)$, $X$ is also separably rationally connected. Since resolution of singularities exists in dimension $2$, by results of Bloch and Srinivas (\cite[Theorem 1 (i)]{Bl-Sr}), we know that $\mathrm{CH}_{alg}^2(X),$ the group of algebraic $1$-cycles algebraically equivalent to $0$ modulo rational equivalence, is weakly representable, i.e. there is an abelian variety $Ab_1(X)$ over $k$ and a universal regular morphism (in the sense of Murre \cite[VIa]{Murre_rep}) $\rho:\mathrm{CH}_{alg}^2(X)\rightarrow Ab_1(X)$ such that $\rho$ induces an isomorphism of groups $\mathrm{CH}_{alg}^2(X_{\Omega})\rightarrow Ab_1(X)(\Omega),$ where $\Omega\supset k$ is a universal domain.\\
\indent Now, let $X\subset \mathbb P^4_k$ be a smooth cubic hypersurface. The linear projection $\mathbb P^4_k\dashrightarrow \mathbb P^2_k$ centered along a general line $l\subset X$ (see for example \cite[Proposition 1.25]{mur_alg_mod_rat}) gives a rational map $X\dashrightarrow \mathbb P^2_k$ which after blowing up $l$, yields an ordinary conic bundle $\widetilde{X}\rightarrow \mathbb P^2_k$. By results of Beauville (\cite[Theorem 3.1 and Proposition 3.3]{Beau_prym}), the Prym variety $A$ associated to the conic bundle is the algebraic representative of $\mathrm{CH}^2_{alg}(\widetilde X)\simeq \mathrm{CH}^2_{alg}(X)$ and $\mathrm{CH}^2_{alg}(X)=A(k)$. The principally polarized abelian variety $A$ obtained by this construction is independent of the choice of a general $[l]$. So we call $J(X):=A$ the intermediate Jacobian of $X$; it is a $5$-dimensional abelian variety endowed with the principal polarization $\theta$ of $A$. By results of Beauville (\cite[Remark 2.7]{Beau_prym}), we know that there an isomorphism of $\mathbb Z_2$-modules with their intersection forms \begin{equation}t:(H^1(J(X),\mathbb Z_2),\theta)\rightarrow (H^3(\widetilde X, \mathbb Z_2), \langle,\rangle_{\widetilde X})\simeq (H^3(X, \mathbb Z_2),\langle,\rangle_X)
\label{isom_int_form}
\end{equation}
%i.e. $\theta$, seen as a Weil pairing, is determined (see for example \cite[Section 16]{Milne_av}) by the intersection form on $H^3(X,\mathbb Z_2)$. 
%Given a smooth quasi-projective $k$-scheme $B$ and a cycle $\Gamma\in \mathrm{CH}^2(B\times_k X)$, the composition $B\rightarrow\mathrm{CH}^2_{alg}(X)\stackrel{\rho}{\rightarrow}J(X)$ (where the first map is of the form $b\mapsto \Gamma_b-\Gamma_{b_0}$) is a morphism of algebraic varieties, called the Abel-Jacobi map associated to $\Gamma$ and denoted $\phi_\Gamma$. 
\indent It is known (see Murre \cite[section VI]{Mur_some}, see also \cite{C-G} over $\mathbb C$) that the Abel-Jacobi morphism associated to the universal $\mathbb P^1$-bundle $P\subset F(X)\times X$ over the variety of lines $F(X)$ of $X$, induces an isomorphism of abelian varieties $\phi_P:Alb(F(X))\simeq J(X)$ where $Alb(F(X))$ is the Albanese variety of $F(X)$, which is defined in this setting as the dual of the Picard variety $\mathrm{Pic}^0(F(X))$.\\
\indent Since $F(X)$ is the zero locus of a regular section of the vector bundle $\mathcal E=Sym_3(E)$ on the grassmannian $Gr(2,5)$, where $E$ is the rank $2$ quotient bundle on $Gr(2,5)$, we have the following exact sequence:
$$0\rightarrow \wedge^4\mathcal E^*\rightarrow \wedge^3\mathcal E^*\rightarrow \wedge^2\mathcal E^*\rightarrow\mathcal E^*\rightarrow \mathcal O_{G(2,5)}\rightarrow \mathcal O_{F(X)}\rightarrow 0$$ given by the Kozsul resolution of the sheaf of ideals of $F(X)$ in $Gr(2,5)$ (the exactness follows from the fact that $F(X)$ has the codimension $4=rank(\mathcal E)$) so, we have a quasi-isomorphism of complexes $\wedge^\cdot\mathcal E^*\simeq \mathcal O_{F(X)}[4]$.\\
\indent Now, we have a spectral sequence $E^{p,q}_1=H^q(Gr(2,5),\wedge^{4-p}\mathcal E^*)\Rightarrow H^{p+q-4}(F(X), \mathcal O_{F(X)})$, which, according to \cite[Theorem 5.1]{Alt-Kl}, degenerates at $E_1$, so that $H^1(F(X), \mathcal O_{F(X)})\simeq H^3(Gr(2,5),\wedge^2\mathcal E^*)$. According to \cite[Proposition 5.11 and Lemma 5.7]{Alt-Kl}, we have an isomorphism $H^3(Gr(2,5),\wedge^2\mathcal E^*)\simeq H^0(\mathbb P_k^4,T\mathbb P^4_k(-1))$. By the Euler sequence: $$0\rightarrow \mathcal O_{\mathbb P^4_k}(-1) \rightarrow H^0(\mathbb P_k^4,\mathcal O_{\mathbb P^4_k}(1))^\vee\otimes\mathcal O_{\mathbb P^4_k}\rightarrow T\mathbb P^4_k(-1)\rightarrow 0,$$ we have $H^1(F(X), \mathcal O_{F(X)})\simeq H^0(\mathbb P_k^4,\mathcal O_{\mathbb P^4_k}(1))^\vee$. Tensoring the normal bundle exact sequence of the inclusion $X\subset \mathbb P^4_k$ with $\Omega^3_{X/k}\simeq \mathcal O_{\mathbb P^4_k}(-2)_{|X}$ and looking at the associated long exact sequence, we have $H^0(X,\mathcal O_X(1))\simeq H^1(X,\Omega_{X/k}^2)$. Since $H^0(X,\mathcal O_X(1))\simeq H^0(\mathbb P_k^4,\mathcal O_{\mathbb P^4_k}(1))$, we have $H^1(F(X), \mathcal O_{F(X)})\simeq H^1(X,\Omega_{X/k}^2)^\vee$. Since by \cite[Theorem 8]{Mur_some}, the Picard and Albanese varieties of $F(X)$ are isomorphic, we have $T_0J(X)\simeq H^2(X,\Omega_{X/k})$.\\

\indent In order to apply Voisin's method in \cite{Vois_abel} to analyse the existence of a decomposition of the diagonal for a cubic threefold, we need to make sure there exists, as it is the case in characteristic $0$ (\cite{Mar-Tik} and \cite{Ili-Mar}), a parametrization of $J(X)$ with separably rationally connected generic fiber, namely condition (*) of the introduction. This parametrization will be constructed in Section $4$. So let us proceed to the proof of the criterion (Theorem \ref{criterion} of the introduction) for the existence of a cohomological decomposition of the diagonal assuming that (*) holds.

\subsection{Decomposition of the diagonal for a smooth cubic threefold}
\indent In this section, we prove the following theorem which was first proved over $\mathbb C$ in \cite[Theorem 4.9]{Vois_abel}, \cite[Theorem 4.1]{Vois_main}. Item (ii) is specific to the finite field situation and is Theorem \ref{cor_intro} of the introduction.
\begin{theoreme}\label{crit_cub} (i) Let $X\subset \mathbb P^4_k$ be a smooth cubic hypersurface ($k=\bar k$ and $char(k)>2$). Then $X$ admits a cohomological (hence Chow-theoretic by Theorem \ref{coh-chow}) decomposition of the diagonal if and only if there is a $\gamma\in \mathrm{CH}_1(J(X))\otimes_{\mathbb Z}\mathbb Z_2$ such that $\theta^4/4!=[\gamma]$ in $H^8(J(X),\mathbb Z_2(4))$.\\
%\indent (ii) Conversely, let $k$ be an algebraically closed field of characteristic $\geq 7$. Then if a smooth cubic hypersurface of $\mathbb P^4_k$ admits a Chow-theoretic decomposition of the diagonal, then the minimal class $\theta^4/4!\in H^8(J(X),\mathbb Z_2(4))$ is algebraic on $J(X)$ i.e. $\exists \gamma\in \mathrm{CH}_1(J(X))\otimes_{\mathbb Z}\mathbb Z_2$ such that $\theta^4/4!=[\gamma]$.\\
\indent (ii) If $k=\overline{\mathbb F_p}$, $p>2$ and the Tate conjecture is true for divisors on surfaces over finite fields, then every smooth cubic hypersurface of $\mathbb P^4_k$ admits a Chow-theoretic decomposition of the diagonal.
\end{theoreme}

\begin{proof} Assume $\theta^4/4!\in \mathrm{CH}_1(J(X))\otimes \mathbb Z_2.$ We will prove that $X$ admits a cohomological (with coefficient in $\mathbb Z_2$) decompositon of the diagonal. We know that $H^*(X,\mathbb Z_2)$ has no torsion so applying K\"unneth decomposition, we can write $[\Delta_X]= \sum_{i=0}^6 \delta_{i,6-i}$ where $\delta_{i,6-i}\in H^i(X,\mathbb Z_2)\otimes H^{6-i}(X,\mathbb Z_2)$ are the components of $[\Delta_X]\in H^6(X\times_k X,\mathbb Z_2).$ Since $H^{1}(X,\mathbb Z_2)=0=H^{5}(X,\mathbb Z_2)$ we have $\delta_{1,5}=0=\delta_{5,1}.$ We know that $\delta_{6,0}$ is the class $X\times_k x$ for any point $x\in X(k)$ and $\delta_{0,6}$ is the class of the subvariety $x\times_k X$ which obviously does not dominate $X$ by the first projection. Since $H^2(X,\mathbb Z_2)$ and $H^4(X,\mathbb Z_2)$ are algebraic, $\delta_{2,4}$ and $\delta_{4,2}$ are linear combinations of classes of algebraic subvarieties of $X\times_k X$ that do not dominate $X$ by the first projection. The existence of a cohomological decomposition of the diagonal with coefficients in $\mathbb Z_2$ is thus equivalent to the existence a cycle $Z\subset X\times_k X$ suth that the support of $Z$ is contained in $D\times_k X,$ with $D\subset X$ a proper subscheme, and $Z^*:H^3(X,\mathbb Z_2)\rightarrow H^3(X,\mathbb Z_2)$ is the identity map since in this case $\delta_{3,3}=[Z]$. We proceed as in \cite[Theorem 4.9]{Vois_abel} to construct such a $Z$.\\
\indent Let $C = \sum_i m_i C_i$ be a $1$-cycle of $J(X)$ of class $\theta^4/4!$ in $H^8(J(X),\mathbb Z_2)$, where $C_i\subset J(X)$ are curves and $m_i\in \mathbb Z_2$. According to condition (*), which is Theorem \ref{p_param}, there are a smooth $10$-dimensional quasi-projective $k$-variety $B$ and a cycle $\mathcal Z\subset B\times_k X,$ flat over $B,$ such that the induced Abel-Jacobi morphism $\phi_{\mathcal Z}:B\rightarrow J(X)$ is dominant with general fiber $\mathbb P^5_k$. Let $B'$ be the closure of the graph of $\phi_{\mathcal Z}$ in $\overline{B}\times J(X)$, where $\overline{B}$ is a compactification of $B$; $B'$ is birational to the quasi-projective variety $B$ and the projection yields a proper, surjective morphism $\phi:B'\rightarrow J(X)$ with general fiber $\mathbb P^5_k$. Let $W\subset J(X)$ be an open subscheme contained in the image of $\phi_{\mathcal Z}$ such that $\forall x\in W(k),\ \phi^{-1}(x)\simeq \mathbb P^5_k.$
Using Chow's moving lemma, we can assume that the generic point of each $C_i$ is in $W.$ Let $n_i:\widetilde{C_i}\rightarrow C_i$ be the normalization of $C_i.$ Let $N_i:B_i\rightarrow B'\times_{C_i}\widetilde{C_i}$ be the normalization morphism. A component of $B_i$ is proper and flat (because integral over a smooth curve) over the smooth curve $\widetilde{C_i}$ with general fiber $\mathbb P^5_k$ so by Tsen's Theorem (\cite{Tsen}), $B_i\rightarrow \widetilde{C_i}$ admits a section $\sigma_i$ and $\phi\circ n_i'\circ N_i\circ\sigma_i = n_i$ (where $n_i': B'\times_{C_i}\widetilde{C_i}\rightarrow B'$ is the projection). Let $Z_i\subset \widetilde{C_i}\times_k X$ be the cycle $(pr_B^{B\times_k J(X)}\circ n_i'\circ N_i\circ\sigma_i,id_X)^*\mathcal Z$. We have the easy equality:
\begin{lemme} The homomorphisms $Z_{i,*}\ and\ t\circ \phi_{Z_i, *}: H^1(\widetilde{C_i},\mathbb Z_2)\rightarrow H^3(X,\mathbb Z_2)$ coincide (where $t:(H^1(J(X),\mathbb Z_2),\theta)\rightarrow (H^3(X,\mathbb Z_2),\langle,\rangle_X)$ is the isomorphism).
\end{lemme}
Let $Z\subset X\times_k X$ be the cycle $\sum_im_iZ_i\circ\,^t Z_i.$ We have $$(Z_i\circ\,^t Z_i)^* = Z_{i,*}\circ Z_i^* = (t^{-1})^*\circ \phi_{Z_i,*}\circ \phi_{Z_i}^*\circ t^{-1}$$ where $(t^{-1})^*$ is the Poincar\'e dual of $t^{-1}.$ But $\phi_{Z_i}$ is just $n_i$ so $\phi_{Z_i,*}\circ \phi_{Z_i}^*$ is $n_{i,*}\circ n_i^*$ which is just $[C_i]\cup:H^1(J(X),\mathbb Z_2)\rightarrow H^9(J(X),\mathbb Z_2).$ Hence $Z^*$ is the composite map $$H^3(X,\mathbb Z_2)\stackrel{t^{-1}}{\simeq} H^1(J(X),\mathbb Z_2)\stackrel{(\theta^4/4!)\cup = \sum_i[C_i]\cup}{\longrightarrow} H^9(J(X),\mathbb Z_2)\stackrel{(t^{-1})^*}{\simeq} H^3(X,\mathbb Z_2).$$ So $Z^*$ is the identity on $H^3(X,\mathbb Z_2)$. On the other hand, $Z$ does not dominate $X$ by the first projection since it is supported on $\bigcup_i \Sigma_i\times \Sigma_i$, where $\Sigma_i$ is the image in $X$ of $\mathrm{Supp}(Z_i)$.\\

\indent We prove now the second direction. Assume $X$ admits a Chow-theoretic decomposition of the diagonal. Then $X$ admits a cohomological decomposition of the diagonal with coefficients in $\mathbb Z_2$ so by Theorem \ref{tech_var_aux} we have finitely many smooth projective curves $Z_i$ and for each curve, a correspondence $\Gamma_i\in \mathrm{CH}^2(Z_i\times_k X)$ and a $n_i\in \mathbb Z_2$ satisfying (\ref{etoile}). The Abel-Jacobi map $\Phi_X$ of $X$ induces (after choosing a reference point in $Z_i$) a morphism $$\gamma_i=\Phi_X\circ \Gamma_{i*}:Z_i\rightarrow J(X)$$ with image $Z_i':= \gamma_{i*}Z_i \in \mathrm{CH}_1(J(X))$. Now we have $\wedge^2H^1(J(X),\mathbb Z_2)\simeq H^2(J(X),\mathbb Z_2)= H^8(J(X),\mathbb Z_2)^*$ and an isomorphism $(H^1(J(X),\mathbb Z_2),\theta)\stackrel{t}{\simeq}(H^3(X,\mathbb Z_2),\langle,\rangle_X)$ with their intersection form (\ref{isom_int_form}). For all $\alpha\in H^1(J(X),\mathbb Z_2),\ \gamma_i^*\alpha = \Gamma_i^*t(\alpha)$ so for all $\alpha,\ \beta\in H^1(J(X),\mathbb Z_2),$
$$\begin{tabular}{ll}
$\langle\sum_i n_i [Z_i'],\alpha\cup\beta\rangle_{J(X)}$ &$=\sum_i n_i \langle[Z_i']\cup\alpha,\beta\rangle_{J(X)}$\\
&$=\sum_i n_i \langle\gamma_i^*\alpha,\gamma_i^*\beta\rangle_{Z_i}$\\
                                            &$=\sum_i n_i \langle\Gamma_i^*t(\alpha),\Gamma_i^*t(\beta)\rangle_{Z_i}$\\
                                            &$=\langle t(\alpha),t(\beta)\rangle_X \mathrm{\ (by\ \ref{etoile})}$\\
                                            &$=\theta(\alpha,\beta)$\\
                                            &$=\langle\frac{\theta^4}{4!},\alpha\cup\beta\rangle_{J(X)}$\\
\end{tabular}$$
hence $\frac{\theta^4}{4!}=\sum_i n_i[Z'_i]$.\\
\indent (ii) If the Tate conjecture is true for divisors on surfaces defined over finite fields, then the theorem of Schoen (\cite{Schoen}) says that the cycle map $\mathrm{CH}_1(J(X))\otimes \mathbb Z_2\rightarrow \bigcup_U H^8(J(X),\mathbb Z_2(4))^U$, where $U$ runs through all open subgroups of $Gal(k/k_{def})$ ($k_{def}$ being a finite field over which $J(X)$ is defined), is surjective. Since $\theta^4$ is algebraic, $\theta^4/4!\in \bigcup_U H^8(J(X),\mathbb Z_2(4))^U$ and we conclude by point (i) of the theorem.
\end{proof}

\section{Parametrization of the intermediate Jacobian of a smooth cubic threefold} The goal of this section is to prove that condition (*) still hold in the positive characteristic setting. Over $\mathbb C$, such a parametrization was achieved using the space of smooth normal elliptic quintics which we do not know to exist a priori in our setting. So we will construct some stable normal elliptic quintics using the lines on the cubic threefold.
\subsection{Some facts on the Fano variety of lines.}\label{reminder_on_fano} Let $X$ be a smooth cubic hypersurface $\mathbb P^4_k$. Since $\mathbb P^3_k$ is separably rationally connected and there is a dominant degree $2$ rational map $\mathbb P^3_k\dashrightarrow X$ and $2\neq char(k)$, $X$ is separably rationally connected.\\
\indent The Fano variety of lines $F(X)=\{[l]\in Gr(2,5),\ l\subset X\}$ of $X$ is a smooth projective surface. Denote by $P\stackrel{p}{\rightarrow} F(X)$ the universal $\mathbb P^1$-bundle and by $q:P\rightarrow X$ the projection on $X.$\\
\indent We collect some results from Murre (\cite{mur_alg_mod_rat}) on the geometry of $F(X)$. Let $\mathcal F_0\subset F(X)$ be the subset defined by $$\mathcal F_0=\{[l]\in F(X),\ \exists K\subset \mathbb P^4_k\ \mathrm{a}\ 2\mathrm{-plane\ s.t.}\  K\cap X = 2l + l'\ \mathrm{as\ divisors\ in}\ K\}.$$ Then $\mathcal F_0$ is a non-singular curve (\cite[Corollary 1.9]{mur_alg_mod_rat}), the lines $[l]\in \mathcal F_0$ are said to be of the second type and those not in $\mathcal F_0$ are said to be of the first type. The subscheme $$\mathcal F'_0=\{[l]\in F(X),\ \exists K\subset \mathbb P^4_k\ \mathrm{a}\ 2\mathrm{-plane\  and}\ l'\in F(X)\ \mathrm{s.t.}\ K\cap X= l + 2l'\}$$ has dimension at most $1$ (\cite[Lemma 1.11]{mur_alg_mod_rat}). For $[l]\in F(X),$ let us denote by $$\mathcal H (l) =\overline{\{[l']\in F(X),\ [l]\neq[l'],\ l\cap l' \neq \emptyset\}}.$$ Then $\mathcal H(l)$ is a curve in $F(X)$.
We have the following properties:

\begin{proposition}\label{at_most_6}\textit{(\cite[(1.17), (1.18), (1.24), (1.25)]{mur_alg_mod_rat})} (i) If $[l]$ is a line of the first type on $X$ and $x\in l$ then there are only finitely many (in fact at most $6$) lines on $X$ through $x.$ Moreover there is no $2$-plane tangent to $X$ in all points of $l.$\\
\indent (ii) There is an nonempty open subscheme $\mathcal U'\subset F(X)$ contained in $F(X)\backslash (\mathcal F_0\cup \mathcal F'_0)$ such that any $[l]\in \mathcal U'$ ($l$ is of the first type) is contained in a smooth hyperplane section of $X$ and $\mathcal H(l)$ is a smooth irreducible curve of genus $11$.
\end{proposition}
By the jacobian criterion, given a hyperplane $H\subset \mathbb P^4_k$, $X\cap H$ is not a smooth cubic surface if and only if there is a $x\in X$ such that $H$ is tangent to $X$ at $x$. Looking at the Gauss map $\mathcal D: X\rightarrow (\mathbb P^4_k)^*,$ (\cite{C-G}) we see that the variety $\mathcal D(X)\subset (\mathbb P^4_k)^*$ parametrizing those hyperplanes tangent to $X$ at some point $x\in X$ is a hypersurface in $(\mathbb P^4_k)^*.$ So the general (parametrized by the open subscheme $(\mathbb P^4_k)^*\backslash \mathcal D(X)$) hyperplane section of $X$ is smooth. Moreover, if $[l]\in \mathcal U',$ we know that $K^*_{l}=\{ [H]\in (\mathbb P^4_k)^*,\ l\subset H\},$ which is a $2$-plane in $(\mathbb P^4_k)^*,$ is not contained in $\mathcal D(X),$ hence $\mathcal D(X)\cap K^*_l$ is of dimension at most $1$ and $K^*_l\backslash (\mathcal D(X)\cap K^*_l)$ is a $2$-dimensional open subscheme. So we see that in fact, for $[l]\in \mathcal U',$ the general hyperplane containing $l$ gives a smooth hyperplane section of $X.$ We have also the following properties:

\begin{proposition}\label{csq_no_plane} (i) For $[l]\in \mathcal U'$, $Im(\mathcal H(l))\ (= \cup_{[l']\in \mathcal H(l)} l'=q(p^{-1}(\mathcal H(l))))$ is not contained in a fixed $2$-plane. So if $[l']\in F(X)$ is a line distinct from $[l]$ such that $l\cap l'\neq \emptyset,$ $\mathcal H(l)$ and $\mathcal H(l')$ have no common component.\\
\indent (ii) Let $h_{l_0}:F(X)\backslash (\mathcal H(l_0)\cup\{[l_0]\})\rightarrow K^*_{l_0}$ be the morphism defined by $[l]\mapsto [span(l,l_0)]$ for $[l_0]\in \mathcal U'.$ Then $h_{l_0}$ is dominant and there is an open subcheme $V_{l_0}\subset \mathcal U'\backslash (\mathcal H(l_0)\cup\{[l_0]\})$ such that for $[l]\in V_{l_0},\ h_{l_0}([l])$ gives a smooth hyperplane section.\\
\indent (iii) For $[l]\in \mathcal U'$, there is an open subscheme $\mathcal U'_{[l]}\subset \mathcal U'$ such that for all $[l']\in \mathcal U'_{[l]},$ $\mathcal H(l')\cap \mathcal H(l)$ is finite. So there is an open subscheme $\mathcal U\subset \mathcal U'$ such that $\forall [l]\in \mathcal U,$ $\mathcal H(l)\cap \mathcal U'$ is a non empty open subscheme of $\mathcal H(l)$.
\end{proposition}
\begin{proof} (i) Suppose there exists $K\subset \mathbb P^4_k$ a $2$-plane such that $q(p^{-1}(\mathcal H(l)))\subset K$. $p^{-1}(\mathcal H(l))$ is a smooth irreducible ruled surface over a curve of genus $11$. Since $X$ is smooth, it cannot contain a $2$-plane, so $q(p^{-1}(\mathcal H(l))$ is contrated into a curve or a set of points in $K\cap X$. Since the ruled surface is irreducible, $q(p^{-1}(\mathcal H(l)))$ is an irreducible closed ($q$ proper) subscheme of $K$ and we have $l\subset q(p^{-1}(\mathcal H(l)))$ so $l=q(p^{-1}(\mathcal H(l))).$ But, by \ref{at_most_6} (i), since $l$ is of the first type, the fiber of $q$ over a point of $l$ is $0$-dimensional. So such $K$ does not exist.\\
\indent Moreover, suppose $[l']\in F(X)\backslash \{[l]\}$ is such that $l'\cap l\neq \emptyset$ and $\mathcal H(l)$ (which is irreducible) is one of the components of $\mathcal H(l'),$ then all (except maybe the other $4$ ones passing through $l\cap l'$) lines that meet $l$ are contained in the $2$-plane $span(l,l'),$ which is impossible.\\
\indent (ii) Since $[l_0]\in \mathcal U',$ the general member of $K^*_{l_0}$ gives a smooth hyperplane section which contains some lines that do not meet $l_0$ (i.e in $F(X)\backslash (\mathcal H(l_0)\cup\{[l_0]\})$), $h_{l_0}$ is dominant and the general fiber is $0$-dimensional of cardinal $<27.$ So $h_{l_0\ |\mathcal U'\backslash \mathcal H(l_0)}$ is still dominant. Let $\mathcal C = K^*_{l_0}\cap \mathcal D(X)$ be the closed subscheme of dimension $\leq 1$ parametrizing singular hyperplane sections containing $l_0.$ Since $h_{l_0}$ is dominant, the closed subscheme $h^{-1}(\mathcal C)$ has dimension $\leq 1.$ So the property is proved with $V_{l_0}=\mathcal U'\backslash (\mathcal H(l_0)\cup h^{-1}(\mathcal C)).$\\
\indent (iii) For $[l_0]\in \mathcal U',$ according to the previous point, there is an open subscheme $V_{l_0}\subset \mathcal U'\backslash (\mathcal H(l_0)\cup\{[l_0]\})$ such that $\forall [l]\in V_{l_0},\ h_{l_0}([l])$ is a smooth hyperplane section, in particular, $\forall [l]\in V_{l_0},\ \mathcal H(l_0)\neq \mathcal H(l)$ (otherwise, all these lines would be contained in the smooth cubic surface and $k$ being algebraically closed, $\mathcal H(l_0)(k)$ is infinite). For the last statement, should they exist, take finitely many $[l_j]\in \mathcal U'$ such that the $\mathcal H(l_j)$ are the irreducible components of the divisors ($\mathcal F_0, \mathcal F'_0...$) removed from $F(X)$ to obtain $\mathcal U'$ and set $\mathcal U = \cap_{j}V_{l_j}$ (and $\mathcal U'=\mathcal U$ if there is no such $l_j$).
\end{proof}

\subsection{Space of normal elliptic quintics}\label{constr_curve} We will construct in this section, singular normal elliptic quintic curves, namely cycles of rational curves (generically) with four components, $3$ lines and a conic. We will use for this the properties of the Fano surface from \ref{at_most_6} and \ref{csq_no_plane}.\\

\indent Take $[l_0]\in \mathcal U'.$ Let $\mathcal C\subset K^*_{l_0}$ the closed subscheme of dimension $\leq 1$ parametrizing the singular hyperplane sections containing $l_0$. Its preimage $h^{-1}_{l_0}(\mathcal C)$ has dimension at most $1$ since there is at least one (hence a $2$-dimensional open set of) smooth hyperplane section containing $l_0$. Let us denote by $(C_i)_{1\leq i\leq m}$ the irreducible components of $h_{l_0}^{-1}(\mathcal C)$ and should they exist such, let us denote by $([l^i])_{1\leq i\leq m}\in \mathcal U^m$  some lines such that $\mathcal H(l^i)\cap C_i$ has dimension $1$. We can choose $[l_1]$ in the ($2$-dimensional) open subscheme $\cap_{i=1}^m\mathcal U'_{[l^i]}\cap V_{l_0}\cap \mathcal U$. Then $\mathcal H(l_1)\cap h_{l_0}^{-1}(\mathcal C) = \cup_{i=1}^m\mathcal H(l_1)\cap C_i\subset \cup_{i=1}^m(\mathcal H(l_1)\cap \mathcal H(l^i)),$ the last set is finite as finite union of finite sets ($[l_1]\in \cap_{i=1}^m\mathcal U'_{[l^i]}$). Since the hyperplane section $S=H_1\cap X$ is smooth (where $H_1=span(l_0,l_1)$) and the lines meeting $l_0$ and $l_1$ are in $S$, $\mathcal H(l_1)\cap \mathcal H(l_0)$ is finite.\\

Since $(\mathcal H(l_1)\cap \mathcal U')\backslash (h_{l_0}^{-1}(\mathcal C)\cup \mathcal H(l_0))$ is an nonempty open subscheme of $\mathcal H(l_1)$ and $S$ contains $27$ lines, we can choose $[l_2]$ in this open subset such that $[l_2]$ is not contained in $S$ (in particular not contained in $H_1$) and $H_2\cap X$ is a smooth cubic surface, where $H_2=span(l_0, l_2)$ and $H_2\neq H_1$ ($l_2$ is not in $S$). We have $\mathcal H(l_2)\neq \mathcal H(l_0)$ (because $[l_1]\in \mathcal H(l_2)$ and is not in $\mathcal H(l_0)$) and $\mathcal H(l_2)\neq \mathcal H(l_1)$ by Proposition \ref{csq_no_plane} (i) so that $\mathcal H(l_2)\backslash (\mathcal H(l_0)\cup \mathcal H(l_1))$ is an nonempty open subscheme of $\mathcal H(l_2)$. Since $H_2\cap X$ contains only finitely many lines, we can take $[l_3]\in \mathcal H(l_2)\backslash (\mathcal H(l_0)\cup \mathcal H(l_1))$ not in that surface and not intersecting $l_2$ at $l_1\cap l_2$ (there are at most $4$ other lines passing through $l_1\cap l_2$ by Proposition \ref{at_most_6} (i)).\\
\indent Letting $H_3=span(l_0,l_3)$, we have $H_1\cap H_2\not\subset H_3$ otherwise the point $l_2\cap l_1$ would be in $H_1\cap H_2 \subset H_3$ and the same with the point $l_2\cap l_3$ so that ($l_2\cap l_1 \neq l_2\cap l_3$) we would have $l_2\subset H_3,$ i.e. $H_3=H_2$ but $l_3\not\subset H_2.$ So $H_1\cap H_2\cap H_3$ is the line $l_0.$
\begin{lemme}\label{no_double_line} A hyperplane section of a smooth hypersurface of degree $\geq 2$ in projective space $\mathbb P^n$ has a zero-dimensional singular set. Hence a hyperplane section of a smooth cubic surface does not contain a double line.
\end{lemme} 
Indeed, the singular locus of a hyperplane section $H\cap X$ of $X$ is the fiber over the point $[H]\in (\mathbb P^{n})^*$ of the Gauss map of $X$ which is given by a base point free ample linear system.\\ 

 Applying the lemma, we see that the residual conic $D$ to $l_0$ in $H_3\cap S$ is the union of two (secant) distinct lines or a nondegenerate conic. We have $D\cap l_2= \emptyset$ otherwise, a point $x\in D\cap l_2$ would be in $H_1\cap H_2\cap H_3$ which is known to be $l_0$ and $l_0\cap l_2=\emptyset.$ As the intersection of the $2$-plane $H_1\cap H_3$ with the line $l_1$ in the $3$-dimensional projective space $H_1,\ (H_1\cap H_3)\cap l_1$ is a complete intersection point of $(H_1\cap H_3)\cap X$ not on $l_0$ (since $l_0\cap l_1=\emptyset$) so $D\cap l_1$ is that point. In particular if $D$ is degenerate, it is not the meeting point of the two components of $D$. The same is true for $D \cap l_3$ which the complete intersection point $(H_1\cap H_3)\cap l_3,$ intersection of the $2$-plane $H_1\cap H_3$ with the line $l_3$ in the $3$-dimensional projective space $H_3.$ So $C = l_1\cup l_2\cup l_3\cup D$ is a locally complete intersection closed subset of $X$ of pure dimension $1$. It is a curve of arithmetic genus $1$, with trivial dualizing sheaf, not contained in any hyperplane and whose intersection with a general hyperplane is a degree $5$ effective zero cycle. It is thus a singular linearly normal elliptic quintic curve. From this construction, we see that we have at least a $6$-dimensional family of such curves: $[l_0],[l_1]$ are chosen in open subsets of $F(X),\ [l_2]$ is chosen in an open subset of the curve $\mathcal H(l_1)$ and $[l_3]$ in an open subset of the curve $\mathcal H(l_2)$.\\
\indent The curve $C$ thus constructed have the following properties:

\begin{proposition}\label{p_coh_curve_1} (i) $h^i(X,\mathcal I_C)=0$ $i\in \{0, 1, 3\}$ and $h^2(X,\mathcal I_C)=1$;\\
\indent (ii) $h^i(X,\mathcal I_C(1))=0$ for all $0\leq i\leq 2$;\\
\indent (iii) $h^0(X,\mathcal I_C(2))=5$, $h^i(X,\mathcal I_C(2))=0$ for $i=1,2$;\\
\indent (iv) $dim_kExt^1(\mathcal I_C,\mathcal O_X(-2))=1$ futhermore the generator of $Ext^1(\mathcal I_C,\mathcal O_X(-2))$ generates $\mathcal Ext^1(\mathcal I_C,\mathcal O_X(-2))$ at any point of $C$.
\end{proposition}
\begin{proof} (i) It follows immediately from the long exact sequence associated to the short exact sequence 
\begin{equation} 0\rightarrow \mathcal I_C\rightarrow \mathcal O_X\rightarrow \mathcal O_C\rightarrow 0 
\label{def_cur}
\end{equation}
the connectedness of $C$, which imposes $h^0(C,\mathcal O_C)=1$, Serre duality on $C$ and the isomorphism $\omega_C\simeq \mathcal O_C$.\\
\indent (ii) By the normalization exact sequence we have the inclusion $H^0(C,\mathcal O_C(l))\hookrightarrow \oplus_i H^0(C_i,\mathcal O_{C_i}(l))$ for any $l\in \mathbb Z$, where $C_i$ are the irreducible components of $C$ which are either a projective line or a non-singular conic lying in a plane. So for $l=-1$, we have $h^0(\mathcal O_C(-1))=0$. Now, applying Riemann-Roch theorem to $C$ (see \cite{Mod_cur} p.83) yields $h^0(C,\mathcal O_C(1))=deg(\mathcal O_C(1))=5$ since the intersection of $C$ with a generic hyperplane has degree $5$. Moreover since $C$ is not contained in a hyperplane, $h^0(\mathcal I_C(1))=0$. So tensoring (\ref{def_cur}) by $\mathcal O_X(1)$ and taking the long exact sequence gives the desired cancelations.\\
\indent (iii) By a projective transformation, we can suppose that $C$ is the curve whose components are: the lines $[A_0A_1],\ [A_1A_2],\ [A_2A_3]$ and the conic given by an equation of the form $Q=\alpha X_4^2 + \beta X_0X_3 +\gamma X_0X_4 + \delta X_3X_4$ (so that it meets $A_0$ and $A_3$) in the plane $X_1=0=X_2$ where $A_0=[1:0\dots:0],\dots,A_3=[0:\dots:1:0]$. Then we can check that the quadrics given by union of hyperplanes $Q_0=X_0X_2,\ Q_1=X_1X_3,\ Q_2=X_1X_4,\ Q_3=X_2X_4$ and the quadric $Q$ form a basis of the space of quadrics containing $C$; so we have $h^0(\mathcal I_C(2))=5$. The cancelation of $h^2(X,\mathcal I_C(2))$ follows immediately from the long exact sequence associated to (\ref{def_cur}) tensorized by $\mathcal O_X(2)$. For $h^1(\mathcal I_C(2))$, we proceed as in \cite[Proposition IV.1.2]{Hulek}: for a generic hyperplane $H$, letting $\Gamma:=C\cap H=\{P_1,\dots P_5\}$ where the $P_i$ are five points whose span is equal to $H$ and $S:=H\cap X$ a smooth cubic surface, we have the following diagram with exact rows and columns:
$$\xymatrix{ &0\ar[d] &0\ar[d] &0\ar[d] & \\
0\ar[r] &\mathcal I_C(1)\ar[r]\ar[d]_{H} &\mathcal O_X(1)\ar[r]\ar[d]_{H} &\mathcal O_C(1)\ar[r]\ar[d]_{H} &0\\
0\ar[r] &\mathcal I_C(2)\ar[r]\ar[d] &\mathcal O_X(2)\ar[r]\ar[d] &\mathcal O_C(2)\ar[r]\ar[d] &0\\
0\ar[r] &\mathcal I_\Gamma(2)\ar[r]\ar[d] &\mathcal O_S(2)\ar[r]\ar[d] &\oplus_{i=1}^5 k_{P_i}\ar[r]\ar[d] &0\\
 &0 &0 &0 &}$$
For any $P_{i_0}$, since $\Gamma$ spans $H$, there is a plane for the form $h_1=span(P_{j_1},P_{j_2},P_{j_3})$ ($j_i\in\{1,\dots5\}\backslash \{i_0\}$) that do not contains $P_{i_0}$. we can also choose a plane $h_2$ containing the last point $P_k$ ($\{k\}= \{1,\dots5\}\backslash\{i_0,j_1,j_2,j_3\}$) and not containing $P_{i_0}$, then the union $h_1\cup h_2$ is a quadric of $H$ that does not contain $P_{i_0}$. Thus $H^0(S,\mathcal O_S(2))\rightarrow \oplus_{i=1}^5 k_{P_i}$ is surjective, so $h^1(S,\mathcal I_\Gamma(2))=0$. This gives the surjectivity of $H^1(X,\mathcal I_C(1))\rightarrow H^1(X,\mathcal I_C(2))$ hence $h^1(\mathcal I_C(2))=0$.\\
\indent (iv) The first terms of the local-to-global spectral sequence used to compute the groups $Ext(\mathcal I_C, \mathcal O_X(-2))$ gives
\begin{equation}
\begin{split}
0\rightarrow H^1(X,\mathcal O_X(-2))\rightarrow Ext^1(\mathcal I_C, \mathcal O_X(-2))\rightarrow H^0(X, \mathcal Ext^1(\mathcal I_C,\mathcal O_X(-2)))\\
\rightarrow H^2(X,\mathcal Hom(\mathcal I_C,\mathcal O_X(-2))).
\end{split}
\label{for_morph}
\end{equation}
By \cite[lemma 1]{schnell}, $\mathcal Ext^1(\mathcal I_C,\mathcal O_X(-2))\simeq det(\mathcal{N}_{C/X})\otimes i^*\underbrace{\mathcal O_X(-2)}_{=\omega_X}$, where the vector bundle $\mathcal{N}_{C/X}$ on $C$ is the normal bundle of the locally complete intersection subcheme $C$ in $X$ and $i:C\hookrightarrow X$ the inclusion so $\mathcal Ext^1(\mathcal I_C,\mathcal O_X(-2)) \simeq i_*\omega_C$ the dualizing sheaf of $C,$ hence $H^0(X,\mathcal Ext^1(\mathcal I_C,\mathcal O_X(-2)))\simeq H^0(C,\omega_C).$ But $\omega_C\simeq \mathcal O_C$ and $C$ is proper and connected so $H^0(C,\omega_C)\simeq k.$ Next we have $H^1(X,\mathcal O_X(-2))=0$ and $\mathcal Hom(\mathcal I_C,\mathcal O_X(-2))\simeq \mathcal O_X(-2)$ (\cite[lemma1]{schnell}) so that $H^2(X,\mathcal Hom(\mathcal I_C,\mathcal O_X(-2)))\simeq H^2(X,\mathcal O_X(-2))$. This last group being zero, we have $Ext^1(\mathcal I_C,\mathcal O_X(-2))\simeq H^0(C,\mathcal O_C) \simeq k.$ It is thus clear that the generator of $Ext^1(\mathcal I_C,\mathcal O_X(-2))$ generates $\mathcal Ext^1(\mathcal I_C,\mathcal O_X(-2))$ at any point of $C$.
\end{proof}

By the Serre construction in codimension $2$ (see for example \cite{schnell}), we can associate, using Proposition \ref{p_coh_curve_1} (iv), to a locally complete intersection curve $C\subset X$ constructed as above, a rank $2$ vector bundle on $X$.

\begin{proposition}\label{serre} For any curve $C\subset X$ constructed as above, there is a unique auto-dual rank $2$ vector bundle $\mathcal E$ on $X$ fitting in the exact sequence: 
\begin{equation}0\rightarrow \mathcal O_X\rightarrow \mathcal E(1)\rightarrow \mathcal I_C(2)\rightarrow 0
\label{extension}
\end{equation} given by any non zero element of $Ext^1(\mathcal I_C,\mathcal O_X(-2))\simeq k$.
\end{proposition}

Vector bundles obtained by this method have the following properties which were proved in \cite[Lemmas 2.1, 2.7, 2.8 and Proposition 2.6]{Mar-Tik} for vector bundles constructed by the same method but starting from smooth normal elliptic quintic curves:
\begin{proposition}\label{p_coh_vect} Let $\mathcal E$ be a rank $2$ vector bundle obtained from a singular linearly normal elliptic quintic curve applying Proposition \ref{serre}. Then we have:\\
\indent (i) $h^0(\mathcal E(-1))=0$, $h^0(\mathcal E)=0$, $h^0(\mathcal E(1))=6$ and $h^i(\mathcal E(-1))=0=h^i(\mathcal E(1))$ $\forall i\geq 1$;\\
\indent (ii) $\mathcal E(1)$ is a stable vector bundle generated by its global sections so that the zero locus of any non zero global section $s$ of $\mathcal E(1)$ is a local complete intersection curve $C_s$ on $X$ with trivial dualizing sheaf, whose ideal sheaf satisfies all the equalities of Proposition \ref{p_coh_curve_1} and fits in an exact sequence $0\rightarrow \mathcal O_X\rightarrow \mathcal E(1)\rightarrow \mathcal I_{C_s}(2)\rightarrow 0$;\\
\indent (iii) $h^0(\mathcal E\otimes \mathcal E)=1$, $h^1(\mathcal E\otimes \mathcal E)=5$, $h^2(\mathcal E\otimes \mathcal E)=0=h^3(\mathcal E\otimes \mathcal E)$;\\
\indent (iv) $h^0(\mathcal N_{C/X})=10$, $h^1(\mathcal N_{C/X})=0$. In particular, the Hilbert scheme $Hilb_{5n}(X/k)$, which parametrizes $1$-dimensional subschemes of degree $5$ and arithmetic genus $1$ in $X$, is smooth of dimension $10$ at the point $[C]$.\\
\indent (v) The morphism $\mathbb P^5\simeq\mathbb P(H^0(\mathcal E(1)))\rightarrow Hilb_{5n}(X/k)$, that associates to any non-zero global section of $\mathcal E(1)$ the subscheme of $X$ defined by zero locus, is injective.  
\end{proposition}
\begin{proof} (i) The computation of the dimension of these cohomology groups is immediate from (\ref{extension}) and Proposition \ref{p_coh_curve_1} using long exact sequences with appropriate twists and Serre's duality.\\
\indent (ii) We have seen that $h^1(\mathcal E)=0$, $h^2(\mathcal E(-1))=0$ and $h^3(\mathcal E(-2))=h^0(\mathcal E)=0$ so by Mumford-Castelnuovo criterion, $\mathcal E(1)$ is generated by its global sections. There is no assumption on the characteristic of the base field in the arguments used in \cite[Proposition 2.6]{Mar-Tik} to prove the stability of $\mathcal E(1)$ so the proof adapts to our setting.\\
\indent The proofs of items (iii) and (iv) given in \cite[Lemma 2.7]{Mar-Tik} adapt in positive characteristic since it only uses the fact that stability of vector bundles implies their simplicity, Grothendieck-Riemann-Roch theorem and the fact that $h^1(\mathcal N^*_{C/\mathbb P^4_k}(2))=0$ which is still true in our setting since the second proof of this fact given in \cite[Proposition V.2.1]{Hulek} makes no use of the characteristic ($\neq 2$) nor of a better regularity than local complete intersection. Item (v) is \cite[Lemma 2.8]{Mar-Tik} whose proof uses no assumption on the charateristic of the base field.
\end{proof}

\textit{}\\ \\
\indent Define the locally closed subscheme $\mathcal H$ of $Hilb_{5n}(X/k)$: 
\begin{equation}
\mathcal H=\left\{
\begin{aligned}
&[Z]\in Hilb_{5t}(X/k),\ (i)\ Z\ \mathrm{is\ a\ locally\ complete\ intersection\ of\ pure}\\
& \mathrm{dimension}\ 1, (ii)\ h^1(\mathcal I_Z)=0=h^0(\mathcal I_Z(1))=h^1(\mathcal I_Z(1))\ (hence\ h^0(\mathcal O_Z)=1),\ \\
&(iii)\ h^1(\mathcal I_Z(2))=0=h^2(\mathcal I_Z(2))\ (hence\ h^0(\mathcal I_Z(2)=5),\ (iv)\ \omega_Z\simeq \mathcal O_Z
\end{aligned}
\right\}.
\end{equation}
The $6$-dimensional family of singular linearly normal elliptic quintic curves constructed in the previous section is contained in $\mathcal H_s$. Moreover, the subschemes parametrized by $\mathcal H$ are connected (since $h^0(\mathcal O_Z)=1$) local complete intersection curves with trivial dualizing sheaf and whose ideal sheaves have all the properties needed to guarantee that the coherent sheaf arising from Serre's construction in codimension $2$ (see Proposition \ref{serre}) is a rank $2$ vector bundle on $X$ satisfying the properties of Proposition \ref{p_coh_vect}. In particular, by item (iv) of Proposition \ref{p_coh_vect}, $\mathcal H$ is a smooth $10$-dimensional quasi-projective $k$-scheme. So, using the pull-back $\mathcal Z$ of the universal sheaf over $Hilb_{5n}(X/k)$ on $\mathcal H$, we can define an Abel-Jacobi morphism $\phi_{\mathcal Z}:\mathcal H\rightarrow J(X)$. We have the following theorem which proves that the condition (*) of the introduction is satisfied by cubic threefolds: 

\begin{theoreme}\label{p_param} $\phi_\Gamma$ is smooth and its general fiber is isomorphic to $\mathbb P^5_k$.
\end{theoreme}
\begin{proof} To prove the smoothness of $\phi_{\mathcal Z}$, we just have to see that the differential of the Abel-Jacobi morphism $T\phi_{\mathcal Z}:T_{[C]}\mathcal H\simeq H^0(C,\mathcal N_{C/X})\rightarrow T_{\phi_{\mathcal Z}([C])}J(X)\simeq H^2(X,\Omega_{X/k})$ is surjective for any $[C]\in \mathcal H$. To do so, we proceed as in \cite[Theorem 5.6]{Mar-Tik} trying the prove that the dual of this map $(T\phi_{\mathcal Z})^*:H^2(X,\Omega_{X/k})^*\simeq H^1(X,\Omega_{X/k}^2)\rightarrow H^0(C,\mathcal N_{C/X})^*$ is injective. We use the technique of "tangent bundle sequence" following \cite[Section 2]{We}. It is presented there for a smooth subvariety $C\hookrightarrow X$ and in characteristic zero but the arguments to prove the following lemma make no use, for the subvariety $C$, of a greater regularity than local complete intersection (so that $\mathcal N_{C/X}$ and $\mathcal N_{C/\mathbb P^4_k}$ are vector bundles on $C$ and we can use Serre duality) nor of the characteristic ($\neq 2$). So $(T\phi_{\mathcal Z})^*$ admits the following description:\\
\begin{lemme}\textit{(\cite[Lemma 2.8]{We}).} The following diagram is commutative
$$\xymatrix{H^0(X,\mathcal N_{X/\mathbb P^4_k}\otimes \omega_X)\ar[r]^R\ar[d]_{r_C} &H^1(X,\Omega^2_{X/k})\ar[d]^{(T\phi_{\mathcal Z})^*}\\
H^0(C,\mathcal N_{X/\mathbb P^4_k}\otimes \omega_{X|C})\ar[r]^{\beta_C} &H^0(C,\mathcal N_{C/X})^*}$$
where $r_C$ is the restriction map to $C$, $\beta_C$ is part of the exact sequence $$\cdots\rightarrow H^0(C,\mathcal N_{C/\mathbb P^4_k}\otimes \omega_X)\rightarrow H^0(C,\mathcal N_{X/\mathbb P^4_k}\otimes \omega_{X|C})\stackrel{\beta_C}{\rightarrow}H^1(C,\mathcal N_{C/X}\otimes \omega_X)\simeq H^0(C,\mathcal N_{C/X})^*$$ from the long exact sequence arising from the short exact sequence $$0\rightarrow\omega_X\otimes \mathcal N_{C/X}\rightarrow \omega_X\otimes \mathcal N_{C/\mathbb P^4_k}\rightarrow \omega_X\otimes \mathcal N_{X/\mathbb P^4_k |C}\rightarrow 0$$ and $R$ is the first connecting morphism in the long exact sequence associated to $$0\rightarrow \Omega_{X/k}^2\otimes \mathcal N_{X/\mathbb P^4_k}^*\rightarrow \Omega_{\mathbb P^4_k/k |X}\rightarrow \omega_X\rightarrow 0$$ coming from the exterior cube of the short exact sequence $0\rightarrow \mathcal N_{X/\mathbb P^4_k}^*\rightarrow \Omega_{\mathbb P^4_k|X}\rightarrow \Omega_{X/k}\rightarrow 0$. 
\end{lemme}
Now, $R$ is an isomorphism since Bott formula for $\mathbb P^4_k$ imply the vanishing of $H^0(X,\Omega^3_{\mathbb P^4_k/k|X}(3))$ and $H^1(X,\Omega^3_{\mathbb P^4_k/k|X}(3))$. The kernel of $r_C$ is $H^0(X,\mathcal N_{X/\mathbb P^4_k}\otimes \omega_X\otimes \mathcal I_{C/X})=H^0(X,\mathcal I_{C/X}(1))$ which is $0$ by assumption. As for the kernel of $\beta_C$, we have already seen in the proof of Proposition \ref{p_coh_vect}, that $H^1(C,\mathcal N^*_{C/X}(2))=H^0(C, \mathcal N_{C/X}(-2))$ is $0$ since the second proof given in \cite[Proposition V.2.1]{Hulek} for the vanishing of $h^1(C,\mathcal N^*_{C/\mathbb P^4_k}(2))=h^0(C,\mathcal N_{C/\mathbb P^4_k}(-2))$ makes no use of the characteristic $\neq 2$ nor of a greater regularity than local complete intersection and $\mathcal N_{C/X}(-2)\subset \mathcal N_{C/\mathbb P^4_k}(-2)$; so $\beta_C$ is also injective. Hence $T\phi_{\mathcal Z}$ is surjective at all points of $\mathcal H$ i.e. $\phi_{\mathcal Z}$ is smooth on $\mathcal H$.\\
\indent Since $\phi_{\mathcal Z}$ is smooth, any nonempty closed fiber of $\phi_{\mathcal Z}$ is a disjoint union of smooth $k$-varieties. In fact it is the union of copies of $\mathbb P^5_k$: using the curve $C$ parametrized by a point $[C]$ of this fiber, we can construct a rank $2$ vector bundle $\mathcal E$ having the properties of Proposition \ref{p_coh_vect}; in particular $\mathbb P(H^0(\mathcal E(1)))\simeq \mathbb P^5_k$ (\ref{p_coh_vect} (i)) and the inclusion of \ref{p_coh_vect} (v) $\mathbb P(H^0(\mathcal E(1))) \hookrightarrow Hilb_{5n}(X/k)$ has values in $\mathcal H$ since, by \ref{p_coh_vect} (ii), the curves defined by the zero locus of elements of $\mathbb P(H^0(\mathcal E(1)))$ satisfy the conditions defining $\mathcal H$. Since a morphism from a projective space to an abelian variety is constant, we see that $\phi_{\mathcal Z}(\mathbb P(H^0(\mathcal E(1))))=\phi_{\mathcal Z}([C])$ i.e. $\mathbb P^5_k=\mathbb P(H^0(\mathcal E(1)))$ is included in the fiber $\phi^{-1}_{\mathcal Z}(\phi_{\mathcal Z}([C]))$ and $[C]\in \mathbb P^5_k=\mathbb P(H^0(\mathcal E(1)))$ (by exact sequence \ref{extension}). So every nonempty closed fiber of $\phi_{\mathcal Z}$ is the disjoint union of $\mathbb P^5_k$.\\

\indent Now, let $\pi: \mathfrak X\rightarrow S=Spec(R)$ be a smooth projective lifting of the smooth cubic $X\subset \mathbb P^4_k$ to characteristic $0$ over a DVR. Let us denote by $K=Frac(R)$, the generic point of $S$ and $s$ the closed point. We have also an abelian scheme over $S,\ \mathcal J = Pic^0(\mathcal F(\mathfrak X)/S),$ the relative Picard scheme of the relative Fano surface of $\pi$, whose geometric fibers are isomorphic to the intermediate jacobian of the corresponding smooth cubic hypersurfaces. Let $\mathcal H(\mathfrak X/S)$ be the locally closed subscheme of the relative Hilbert scheme $Hilb_{5t}(\mathfrak X/S)$ defined as
\begin{equation}
\left\{
\begin{aligned}
&[Z]\in Hilb_{5t}(\mathfrak X/S),\ (i)\ Z\ \mathrm{is\ a\ locally\ complete\ intersection\ of\ pure}\\
& \mathrm{dimension}\ 1, (ii)\ h^1(\mathcal I_Z)=0=h^0(\mathcal I_Z(1))=h^1(\mathcal I_Z(1))\ (hence\ h^0(\mathcal O_Z)=1),\ \\
&(iii)\ h^1(\mathcal I_Z(2))=0=h^2(\mathcal I_Z(2))\ (hence\ h^0(\mathcal I_Z(2)=5),\ (iv)\ \omega_Z\simeq \mathcal O_Z
\end{aligned}
\right\}.
\end{equation} whose fiber over $s$ is $\mathcal H$ and fiber over $K$ is the subscheme of the Hilbert scheme $Hilb_{5n}(\mathfrak X_{\eta}/ K)$ defined likewise which, according to \cite[Corollary 5.5]{Mar-Tik}, is also smooth of dimension $10$. It is easy to see that a singular linearly normal elliptic quintic curve, as those constructed in \ref{constr_curve}, lifts in charateristic $0$ over $R$ so $\mathcal H(\mathfrak X/S)$ has a $S$-point that we can use as a reference point to define a morphism $\phi_{\mathcal Z_S}:\mathcal H(\mathfrak X/S)\rightarrow \mathcal J$ (using the universal family $\mathcal Z_S$ of $Hilb_{5t}(\mathfrak X/S)$) inducing the Abel-Jacobi morphisms over $s$ and $K$. 

%According to \cite[Theorem 5.6]{Mar-Tik}, the morphism induced by $\phi_{\mathcal Z_S}$ over $\eta$ is generically smooth. So, there is a open subscheme $U_S\subset \mathcal J$ surjective on $S$ such that $\phi_{\mathcal Z_\eta}$ is smooth over $U_\eta$ and surjective on it and $U_s$ is a dense open subscheme of $J(X)$ over which $\phi_{\mathcal Z}$ is also smooth and surjective. 
Let $\Gamma_{\phi_{\mathcal Z_S}}$ be the closure of the graph of $\phi_{\mathcal Z_S}$ in the $S$-projective scheme $Hilb_{5n}(\mathfrak X/S)\times_S\mathcal J$. Then the second projection $p_S:\Gamma_{\phi_{\mathcal Z_S}}\rightarrow \mathcal J$ yields a projective morphism inducing the Abel-Jacobi morphism $\phi_{\mathcal Z_S}$ on a dense open subscheme of $\Gamma_{\phi_{\mathcal Z_S}}$ which is surjective on $S$. By Stein factorization theorem, there is a proper $S$-scheme $\mathcal M$ such that $p_S$ factorizes as $\Gamma_{\phi_{\mathcal Z_S}}\stackrel{\Phi}{\rightarrow}\mathcal M\stackrel{\Psi}{\rightarrow} \mathcal J$ with $\Psi$ a finite morphism and $\Phi$ a morphism with connected fibers. By work of Iliev, Markutchevich and Tikhomirov (\cite[Theorem 5.6]{Mar-Tik} and \cite[Theorem 3.2]{Ili-Mar}; see also Druel \cite[Theorem 1]{Druel}), the general fiber of $p_K$ is $\mathbb P^5_K$ so that $\Psi_K:\mathcal M_K\rightarrow \mathcal J_k$ is an isomorphism. So $\Psi$ is a birational morphism. Let us denote $\sigma:\mathcal J\dashrightarrow  \mathcal M$ the inverse map. Since $\mathcal J$ is regular, the local ring of any codimension $1$ point of $\mathcal J$ is a 
DVR and since $\mathcal M$ is proper, $\sigma$ is defined on any codimension $1$ point. In particular, it is defined at the generic point of $J(X)=\mathcal J_s$; hence $J(X)$ is birational to a component $M_0$ of $\mathcal M_s$. Let us denote $B=\Phi_s^{-1}(M_0)$. Then $p_s:B\rightarrow J(X)$ is dominant with general fiber isomorphic to $\mathbb P^5$ and on a open subset of $B$, it is the Abel-Jacobi morphism given by the pull-back of the universal family of $Hilb_{5t}(X/k)$.
\end{proof}

\section*{Acknowledgments}
I am grateful to my advisor Claire Voisin for introducing me to this topic and for her patient guidance. I also heartly thank Alena Pirutka for her invaluable technical assistance and her avaibility during this work. I am grateful to the referee for his careful reading and criticism that has help to improve the paper. Finally, I am grateful to the gracious Lord for His care.

\noindent \begin{tabular}[t]{l}
\textit{rene.mboro@polytechnique.edu}\\
CMLS, Ecole Polytechnique, CNRS, Universit\'e Paris-Saclay\\
91128 PALAISEAU C\'edex,\\
FRANCE.\\
\end{tabular}\\
\end{document}